\documentclass[journal]{IEEEtran} 

\usepackage[english]{babel}
\usepackage[latin1]{inputenc}
\usepackage{enumerate}
\usepackage{color}
\usepackage[T1]{fontenc}
\usepackage{epstopdf}
\usepackage{subfigure}
\usepackage{dsfont}
\usepackage[final]{graphicx}
\DeclareGraphicsExtensions{.eps}
\usepackage[T1]{fontenc}
\usepackage{amsmath}
\usepackage{mathtools}
\usepackage{amsthm}
\usepackage{amstext}
\usepackage{amssymb}
\usepackage{mathrsfs}
\usepackage{cite}
\usepackage{mathtools}
\usepackage{tikz}
\usetikzlibrary{arrows,positioning}

\def\R{\mathbb{R}}

\def\rP{\mathbb{P}}

\def\argmax{\mathop{\rm arg\, max}}

\def\X{{\mathcal X}}

\def\S{{\mathcal S}}

\def\bdelta{{\boldsymbol \delta}}
\def\bxi{{\boldsymbol \xi}}
\def\bgamma{{\boldsymbol \gamma}}
\def\bGamma{{\boldsymbol \Gamma}}

\def\bu{{\bf u}}
\def\by{{\bf y}}

\def\rB{{\mathsf{B}}}

\def\reg{{\mathsf{reg}}}

\def\sX{{\mathsf X}}
\def\sY{{\mathsf Y}}

\def\sE{{\mathsf E}}

\def\sU{{\mathsf U}}

\theoremstyle{remark}

\newtheorem{definition}{Definition}
\newtheorem{theorem}{Theorem}
\newtheorem{corollary}{Corollary}

\newtheorem{proposition}{Proposition}
\newtheorem{lemma}{Lemma}
\theoremstyle{remark}

\newtheorem{assumption}{Assumption}

\IEEEoverridecommandlockouts

\allowdisplaybreaks

\begin{document}
\sloppy
\title{Regularized Stochastic Team Problems
\author{Naci Saldi
\thanks{The author is with the Department of Natural and Mathematical Sciences, Ozyegin University, Cekmekoy, Istanbul, Turkey, Email: naci.saldi@ozyegin.edu.tr}
}}
\maketitle

\begin{abstract}
In this paper, we introduce regularized stochastic team problems. Under mild assumptions, we prove that there exists an unique fixed point of the best response operator, where this unique fixed point is the optimal regularized team decision rule. Then, we establish an asynchronous distributed algorithm to compute this optimal strategy. We also provide a bound that shows how the optimal regularized team decision rule performs in the original stochastic team problem. 
\end{abstract}

\begin{keywords}                           
Team decision theory, decentralized stochastic control, regularization, person-by-person optimality. 
\end{keywords}

\section{Introduction}\label{sec1}

Team decision theory has been introduced by Marschak \cite{mar55} to study decisions of agents that are acting collectively based on their private information to optimize a common reward function. Radner \cite{rad62} proved fundamental results for static teams and in particular established connections between Nash equilibrium and team-optimality. Witsenhausen's seminal papers \cite{wit71,wit75,wit88,WitsenStandard,WitsenhausenSIAM71,Wit68} on characterization and classification of information structures have been crucial in the progress of our understanding of teams. In particular, the celebrated counterexample of Witsenhausen \cite{Wit68} demonstrated the challenges that arise due to a decentralized information structure in teams. We refer the reader to \cite{YukselBasarBook} for a more comprehensive overview of team decision theory and a detailed literature review.

In teams, due to its decentralized nature, computing the optimal decision rule is a NP-hard problem \cite{PaTs86}. Indeed, even  
establishing the existence and structure of optimal policies is a challenging problem. Existence of optimal policies for static teams and a class of sequential dynamic teams has been shown recently in \cite{GuYuBaLa15,Sal20-D,YukselWitsenStandardArXiv}. For a class of teams which are {\it convex}, one can reduce the search space to a smaller parametric class of policies (see \cite{rad62,KraMar82,WaSc00}, and for a comprehensive review, see \cite{YukselBasarBook}).

In the literature, there are mainly three approaches to compute the optimal or sub-optimal team decision rules \cite{CDCTutorial}: (i) the common information approach \cite{NayyarBookChapter,NayyarMahajanTeneketzis}, (ii) the designer's approach \cite{WitsenStandard,MahajanThesis,MahTen09,YukselWitsenStandardArXiv}, and (iii) the person-by-person approach \cite{rad62,marrad72}.
In the common information approach, it is assumed that agents share some common information with each other (i.e., delayed observation sharing or periodic observation sharing). Therefore, one can partition the information of each agent into two as the common information and the private information. In other words, there is a coordinator that observes the common information and shares this information with other agents. In the common information approach, the idea is to formulate the problem as a centralized stochastic control problem from the viewpoint of a coordinator. In this centralized control problem, coordinator observes the common information and chooses controllers, where these controllers maps private information of each agent to their actions. With this viewpoint, one can then use classical stochastic control techniques (such as dynamic programming) to compute the optimal team decision rule.

The designer's approach is very similar to the common information approach. Namely, although the original problem has decentralized information structure, it is a centralized decision
problem from the viewpoint of a system designer that (centrally) chooses the policies of all the agents. Hence, one can obtain a dynamic programming recursion of this centralized decision problem by identifying an appropriate information state for the designer. However, in this approach, the action space of the designer is in general too large, and so, computing the optimal policy is mostly unfeasible.

The person-by-person approach is the technique that is adopted from game theory for computing the Nash equilibrium. This approach can be described as follows. Fix policies of all agents except Agent~$i$ and consider the sub-problem of optimally choosing the best policy of Agent~$i$ against the policies of other agents. This is indeed a centralized control problem as the policies of other agents are fixed. Hence, one can use classical stochastic control techniques to arrive at the best response policy of Agent~$i$. Iterating in this manner for each agent, the computed policies eventually converge to the Nash equilibrium. However, although optimal team decision rule is a Nash equilibrium, in team problems, there are in general more than one Nash equilibrium (sometimes we may have infinitely many Nash equilibria), and so, this procedure mostly converges to a sub-optimal team decision rule (see \cite{Bas87} and \cite[Section 2.6]{YukselBasarBook} for the conditions of existence of the unique Nash equilibrium). Note that if one introduce some regularization term to the reward function, it is possible to prove that there exists an unique Nash equilibrium. If the optimal team decision rule exists, this unique Nash equilibrium is also the unique optimal team decision rule. Therefore, the person-by-person approach converges to the optimal team decision rule. This is indeed the approach adapted in this paper. 

Note that if there is a misspecification in decentralized control models, above-mentioned algorithms often results in policies that are far from optimal or sub-optimal. This is due to lack of continuity of the reward function and the optimal policy with respect to the components (i.e., observation channels) of the problem. Therefore, making use of regularization provides a way to overcome this robustness problem. Most recent learning algorithms for control problems also use regularization to increase robustness, and this regularization is generally established via entropy or relative entropy. We refer the reader to \cite{GeScPi19} for an exhaustive review of the literature on regularized centralized stochastic control problems and \cite{NeJoGo17} for a general framework on entropy-regularized centralized stochastic control problems.

In this paper, we introduce regularized team problems with finite observation and action spaces, analogous to regularized centralized stochastic control problems.  We introduce regularization as an additive term to the reward function. We then define best response operator on the set of policies as described above and prove that it has an unique fixed point, where this unique fixed point is the unique optimal team decision rule. Then, we establish an asynchronous distributed algorithm to compute this optimal policy. Note that introducing  the regularization term is supposed to make algorithm more robust against modeling uncertainties. We also provide a bound that shows how the optimal regularized team decision rule performs in the original team problem. Therefore, the solution of the regularized team problem provides an upper bound and a lower bound to the original team problem, which can be used to analyze the performance of certain numerical algorithms developed in the literature.

The paper is organized as follows. In Section~\ref{unregularized}, we formulate classical stochastic dynamic team model and its static reduction. In Section~\ref{regularized}, we introduce regularized stochastic teams. In Section~\ref{bestresponse}, we define best response operator for regularized teams and prove the existence of optimal regularized team decision rules via establishing uniqueness of the fixed point of the best response operator. In Section~\ref{asyiteration}, we propose an asynchronous distributed algorithm to compute the optimal regularized team decision rule.
Section~\ref{conc} concludes the paper.

\section{Unregularized Stochastic Team Model}\label{unregularized}

To model unregularized stochastic teams, we use Witsenhausen's {\it intrinsic model} \cite{wit75}. In this model, we have the following components:
$$
\bigl\{ (\sX,\X), \rP, \sU_i, \sY_i, i=1,\ldots,N\bigr\} \nonumber
$$
where $(\sX,\X)$ is a Borel space (Borel subset of complete and separable metric space) denoting the state space, the finite sets\footnote{We note that a stochastic team with the continuous observation and action spaces can be asymptotically approximated by stochastic teams with finite observation and action spaces \cite{SaYuLi17-app}. Therefore, results developed in this paper can provide insights about the near optimal policies for stochastic teams with continuous observation and action spaces.} $\sU_i$ and $\sY_i$ ($i=1,\ldots,N$) denote the action space and the observation space of Agent~$i$, respectively. Here $N$ is the number of agents. For each $i$, the observations and actions of Agent~$i$ are denoted by $Y_i$ and $U_i$, respectively. The $\sY_i$-valued observation variable for Agent~$i$ is given by 
$$Y_i \sim W_i(\cdot|X, {\bf U}^{^{(1:i-1)}}),$$ 
where $W_i$ is a conditional probability on $\sY_i$ given $\sX\times \sU^{^{(1:i-1)}}$, where $\sU^{^{(1:i-1)}} \coloneqq \prod_{k=1}^{i-1} \sU_k$. A {\it Borel probability measure} $\rP$ on $\sX$ describes the uncertainty on the state variable $X$.

A joint control strategy $\bgamma= (\gamma_1, \gamma_2, \dots, \gamma_N)$, also called {\it policy}, is an $N$-tuple of measurable functions 
$$\gamma_i:\sY_i \to \Delta_i, \, i=1,\ldots,N,$$
where $\Delta_i$ is the set of probability distributions on the finite set $\sU_i$. Here, $\Delta_i$ is endowed with $l_1$-norm $\|\cdot\|_1$ since it can be viewed as a subset of $\R^{\sU_i}$. The $\sigma$-algebra on $\Delta_i$ is the Borel $\sigma$-algebra generated by $l_1$-norm topology. Note that $U_i \sim \gamma_i(\cdot|Y_i)$; that is, $\gamma_i(\cdot|Y_i)$ is a conditional distribution of the action $U_i$ of Agent~$i$ given its observation $Y_i$. Let $\Gamma_i$ denote the set of all admissible policies for Agent~$i$; that is, the set of all measurable functions from $\sY_i$ to $\Delta_i$ and let $\bGamma = \prod_{i=1}^N \Gamma_i$.

Under this intrinsic model, a sequential team problem is {\it dynamic} if the information $Y_i$ available to at least one agent~$i$ is affected by the action of at least one other agent~$k\neq i$. A decentralized problem is {\it static}, if the information available at every decision maker is only affected by state of the nature; that is, no other decision maker can affect the information at any given decision maker.

For any $\bgamma = (\gamma_1, \cdots, \gamma_N)$, we let the
unregularized reward of the team problem be defined by
\[J(\bgamma) \coloneqq E\left[p(X,{\bf Y},{\bf U})\right]\]
for some reward function 
$$p: \sX \times \prod_{i=1}^{N} \sY_i \times \prod_{i=1}^N \sU_i \to [0,\infty),$$ 
where ${\bf U} \coloneqq (U_1,\ldots,U_N)$ and ${\bf Y} \coloneqq (Y_1,\ldots,Y_N)$.

\smallskip

\begin{definition}\label{optimalstrategy}
For a given stochastic team problem, a policy (strategy)
$\bgamma^*:=({\gamma_1}^*,\ldots, {\gamma_N}^*)\in {\bf \Gamma}$ is
an {\it optimal team decision rule} if
\begin{equation}
J(\bgamma^*)=\sup_{\bgamma \in {\bGamma}}
J(\bgamma)=:J^*. \nonumber
\end{equation}
The reward level $J^*$ achieved by this strategy is the {\it optimal value of the team}. A policy (strategy)
$\bgamma^*:=({\gamma_1}^*,\ldots, {\gamma_N}^*)\in {\bf \Gamma}$ is {\it person-by-person optimal} if
\begin{equation}
J(\bgamma^*)=\sup_{\gamma_i \in \Gamma_i}
J(\bgamma^{*,-i},\gamma_i), \text{ } \text{for all} \,\, i=1,\ldots,N, \nonumber
\end{equation}
where $\bgamma^{*,-i} \coloneqq (\gamma_k^*)_{k \neq i}$.
\end{definition}

\subsection{Independent Static Reduction of Dynamic Team Model}\label{staticreduction}

In this section, we review the equivalence between dynamic teams and their independent static reduction (this is called {\it the equivalent model} \cite{wit88}\index{Witsenhausen's Equivalent Model}). Consider a dynamic team setting. Note that, for a fixed choice of $\bgamma$, the reward function $J(\bgamma)$ can be written as
\begin{align}
&J(\bgamma) \nonumber \\
&= \int_{\sX} \, \sum_{\by,\bu} p(x,\by,\bu) \,  \left( \prod_{i=1}^N W_i(y_i|x,\bu^{^{(1:i-1)}}) \, \gamma_i(u_i|y_i) \right)  \hspace{-3pt} \rP(dx) \nonumber \\
%&= \sum_{\by,\bu} \left(\int_{\sX} \biggl[ p(x,\by,\bu) \prod_{i=1}^N |\sY_i| \cdot \, W_i(y_i|x,\bu^{^{(1:i-1)}}) \biggr] \, \rP(dx) \right) \, \prod_{i=1}^N \frac{1}{|\sY_i|} \gamma_i(u_i|y_i) \nonumber \\
&= \sum_{\by,\bu} r(\by,\bu)\,  \prod_{i=1}^N \pi_i(y_i) \,  \gamma_i(u_i|y_i), \nonumber
\end{align}
where
\begin{align}
r(\by,\bu) &\coloneqq \int_{\sX} \biggl[ p(x,\by,\bu) \prod_{i=1}^N |\sY_i| \cdot \, W_i(y_i|x,\bu^{^{(1:i-1)}}) \biggr] \rP(dx) \nonumber,
\end{align}
and $\pi_i$ is uniform distribution on $\sY_i$. Now, the observations can be regarded as independent, and by incorporating the $W_i$ terms into $p$, we can obtain an equivalent {\it independent static team} problem. Hence, the essential step is to appropriately change the probability measure of observations and the reward function. This method is discrete-time version of Girsanov change of measure method. Indeed, a continuous-time generalization of static reduction via Girsanov's method has been presented by Charalambous and Ahmed \cite{charalambous2014equivalence}.

\subsection{Another Equivalent Team Model}\label{equivalent}

In this section, we re-formulate static reduction of dynamic team model as follows. In this model, we take the action space of each Agent~$i$ as the set of probability distributions $\Delta_i$ on the original action space $\sU_i$. Hence, in this new model, strategy $\gamma_i$ of Agent~$i$ can be viewed as a measurable function from the observation space $\sY_i$ to the action space $\Delta_i$ (not a conditional distribution). The new reward function $R: \prod_{i=1}^N \sY_i \times \prod_{i=1}^N \Delta_i \rightarrow [0,\infty)$ is defined as 
$$
R(\by,\bdelta) \coloneqq \sum_{\bu} r(\by,\bu) \, \prod_{i=1}^N \delta_i(u_i). 
$$ 
With these definitions, for a given policy $\bgamma$, its reward function is given by 
$$
J(\bgamma) = \sum_{\by} R(\by,\bgamma(\by)) \, \prod_{i=1}^N \pi_i(y_i),
$$
where $\bgamma(\by) \coloneqq \left(\gamma_1(y_1),\ldots,\gamma_N(y_N)\right)$. It is straightforward to prove that this model is equivalent to the static reduction of the dynamic team model. In the remainder of this paper, we consider this new model.

\subsection{Duality of Strong Convexity and Smoothness}\label{duality}

In this section, before we introduce regularized stochastic team problems, we review duality between strongly convex functions and smooth functions. In particular, we state a key result that will be used to establish the global optimality of the person-by-person optimal strategy and an algorithm for computing this optimal strategy.  

Suppose that $\sE = \R^d$ for some $d \geq 1$ with an inner product $\langle \cdot,\cdot \rangle$. We denote $\R^{*} = \R \, \bigcup \, \{\infty\}$. Let $f:\sE \rightarrow \R^{*}$ be a convex function with the domain $S \coloneqq \{x \in \sE: f(x) \in \R\}$, which is necessarily convex subset of $\sE$. The subdifferential $\partial f(x)$ of $f$ at any point $x \in \sE$ is defined as 
$$
\partial f(x) = \{y\in\sE: f(z) \geq f(x) + \langle y,z-x \rangle \,\, \forall z \in \sE\}.
$$
For any $x \in S$, $\partial f(x) \neq \emptyset$. 
The Fenchel conjugate of $f$ is a convex function $f^*:\sE \rightarrow \R^*$ that is defined as
$$
f^*(y) \coloneqq \sup_{x \in \S} \, \langle x,y \rangle - f(x).
$$

Now, we will state duality result between strong convexity and smoothness. To this end, we suppose that $f$ is $\rho$-strongly convex with respect to a norm $\|\cdot\|$ on $\sE$  (not necessarily Euclidean norm); that is, for all $x,y \in S$, we have 
$$
f(y) \geq f(x) + \langle z,y-x \rangle + \frac{1}{2} \rho \|y-x\|^2 \,\, \text{for all} \,\, z \in \partial f(x). 
$$
To state the result, we need to define the dual norm of $\|\cdot\|$. The dual norm $\|\cdot\|_*$ of $\|\cdot\|$ on $\sE$ is defined as 
$$
\|z\|_* \coloneqq \sup\{\langle z,x \rangle: \|x\| \leq 1\}.
$$
For example, $\|\cdot\|_{\infty}$ is dual norm of $\|\cdot\|_{1}$.  

\smallskip

\begin{proposition}[{\cite[Lemma 15]{Sha07}}]\label{duality}
Let $f:\sE \rightarrow \R^*$ be a $\rho$-strongly convex function with respect to the norm $\|\cdot\|$ and let $S$ denote its domain. Then
\begin{itemize}
\item[1.] $f^*$ is differentiable on $\sE$.
\item[2.] $\nabla f^*(y) = \argmax_{x \in S} \langle x,y \rangle - f(x)$. 
\item[3.] $f^*$ is $\frac{1}{\rho}$-smooth with respect to the norm $\|\cdot\|_*$ ; that is, 
$$
\|\nabla f^*(y_1) - \nabla f^*(y_2)\| \leq \frac{1}{\rho} \|y_1-y_2\|_* \,\, \text{for all} \,\, y_1,y_2 \in \sE.
$$
\end{itemize}
\end{proposition}

In the sequel, we will make use of the properties $2.$ and $3.$ of Proposition~\ref{duality} to establish the Lipschitz continuity of the best-response strategies of agents, which enables us to prove the main results of our paper.

\section{Regularized Stochastic Team Model}\label{regularized}

In this section, we introduce regularized version of the stochastic team model in Section~\ref{equivalent}. Indeed, the only difference between that model and the regularized one is the reward function. The rest of the components and the definitions are the same. To define regularized reward function, for each Agent~$i$, let $\Omega_i: \Delta_i \rightarrow \R$ be a $\rho_i$-strongly convex function with respect to $l_1$-norm $\|\cdot\|_1$. Note that one can extend the definition of $\Omega_i$ to the whole $\R^{\sU_i}$ by setting $\Omega_i(\delta) = \infty$ if $\delta \in \R^{\sU_i} \setminus \Delta_i$. In this case, the domain of $\Omega_i$ becomes $\Delta_i$. In regularized stochastic team model, the reward function is given by
$$
R^{\reg}(\by,\bdelta) \coloneqq R(\by,\bdelta) - \sum_{i=1}^N \Omega_i(\delta_i). 
$$
Here, $\sum_{i=1}^N \Omega_i(\delta_i)$ is the regularization term. 

A canonical example for $\Omega_i$ is the negative of the entropy $\Omega_i(\delta) = \sum_{u \in \sU_i} \ln(\delta(u)) \, \delta(u)$. A similar example is the relative entropy between $\delta$ and uniform distribution; that is, $\Omega_i(\delta) = \sum_{u \in \sU_i} \ln(\delta(u)) \, \delta(u) + \ln(|\sU_i|)$. In both of these examples, as a result of entropy regularization, agents are enforced to apply optimal and as well as almost optimal actions randomly. This improves the exploration of the iterative algorithms for computing (almost) optimal policies, and so, pretends agents to be stuck with local optima. Moreover, due to strong convexity of $\Omega_i$ for each $i=1,\ldots,N$, Lipschitz sensitivity of the optimal policy of Agent~$i$ on observations of agents, policies of other agents, and other uncertain parameters can be established via duality between strong convexity and smoothness. This makes the problem robust to uncertainties in the model. Therefore, if there is a slight misspecification in the model, then, as a result of the regularization, the computed optimal policy remains close to the true optimal policy \cite[Remark 4.3]{KaYu20},\cite[Theorem 4.1]{KaYu19}. This is generally not the case in the original problem as the optimal policy is quite sensitive to the system components.    

In regularized stochastic teams, the reward of a strategy $\bgamma$ is given by 
$$
J^{\reg}(\bgamma) \coloneqq \sum_{\by} R^{\reg}(\by,\bgamma(\by)) \, \prod_{i=1}^N \pi_i(y_i).
$$
With this definition, we can now define notion of optimality.

\smallskip

\begin{definition}\label{optimalregularizedstrategy}
For a given stochastic team problem, a policy (strategy)
$\bgamma^*:=({\gamma_1}^*,\ldots, {\gamma_N}^*)\in {\bf \Gamma}$ is
an {\it optimal regularized team decision rule} if
\begin{equation}
J^{\reg}(\bgamma^*)=\sup_{\bgamma \in {\bGamma}}
J^{\reg}(\bgamma)=:J^{*,\reg}. \nonumber
\end{equation}
The reward level $J^{*,\reg}$ achieved by this strategy is the {\it optimal regularized value of the team}. A policy (strategy)
$\bgamma^*:=({\gamma_1}^*,\ldots, {\gamma_N}^*)\in {\bf \Gamma}$ is {\it regularized person-by-person optimal} if
\begin{equation}
J^{\reg}(\bgamma^*)=\sup_{\gamma_i \in \Gamma_i}
J^{\reg}(\bgamma^{*,-i},\gamma_i), \text{ } \text{for all} \,\, i=1,\ldots,N. \nonumber
\end{equation}
\end{definition}

In this paper, we first show that there exists an unique regularized person-by-person optimal policy. Since an optimal regularized team decision rule is also a regularized person-by-person optimal, this result implies that this unique regularized person-by-person optimal policy must be globally optimal for the regularized stochastic team problem if the optimal regularized policy exists. Then, we introduce an asynchronous iterative algorithm to compute this optimal policy using best-response maps. We establish that this algorithm converge to the optimal policy. 

\section{Best Response Operator}\label{bestresponse}

In this section, we introduce best response operator and establish that there exists an unique fixed point of this operator, where this unique fixed point is proved to be optimal regularized team decision rule. 

For each $j=1,\ldots,N$, we define
$$
\sup_{\by,\bu^{-j}} \left[\sup_{u_j} r(\by,\bu) - \inf_{u_j} r(\by,\bu)\right] \eqqcolon \lambda_j(r),
$$
where $\bu^{-j} \coloneqq (u_i)_{i\neq j}$. Here, $\lambda_j(r)$ gives the local oscillation of the function $r$ with respect to $u_j$. 

\smallskip

\begin{lemma}\label{localosc}
Given $\by$, we have the following bound
$$
|R(\by,\bdelta) - R(\by,\bxi)| \leq \sum_{i=1}^N \frac{\lambda_i(r)}{2} \, \|\delta_i-\xi_i\|_{1}.
$$
\end{lemma}

\begin{proof}
The proof is given in Appendix~\ref{app1}.
\end{proof}

Note that one can view $\Gamma_i = \{\gamma_i:\sY \rightarrow \Delta_i\} = \Delta_i^{\sY_i}$ as a subset of $\R^{\sU_i\times\sY_i}$ since $\Delta_i \subset \R^{\sU_i}$. Let $\|\cdot\|_{L_1}$ denote the $L_1$-norm on $\R^{\sU_i\times\sY_i}$ (do not confuse this with $l_1$ norm $\|\cdot\|_1$ on $\R^{\sU_i}$) defined as 
$$
\|\gamma_i\|_{L_1} = \sum_{y_i,u_i} |\gamma_i(u_i|y_i)| \, \pi_i(y_i) = \sum_{y_i} \|\gamma_i(\cdot|y_i)\|_1 \pi_i(y_i). 
$$
As a subset of $\R^{\sU_i\times\sY_i}$, $\Gamma_i$ is convex and closed with respect to $L_1$-norm.

To define best response operator $\rB: \bGamma \rightarrow \bGamma$, for any $j=1,\ldots,N$, we first define $\rB_j: \bGamma^{-j} \rightarrow \Gamma_j$ as follows
\begin{align}
&\rB_{j}(\bgamma^{-j})(y_j) \nonumber \\
&= \argmax_{\delta_j \in \Delta_j} E\left[R^{\reg}({\bf Y},\bgamma^{-j}({\bf Y}^{-j}),\delta_j) \bigg | Y_j = y_j \right] \nonumber \\
&= \argmax_{\delta_j \in \Delta_j} \sum_{\by^{-j}} R^{\reg}(\by,\bgamma^{-j}(\by^{-j}),\delta_j) \prod_{i \neq j} \pi_i(y_i) \nonumber \\ 
&\overset{(I)}{=} \argmax_{\delta_j \in \Delta_j} \sum_{\by^{-j}} R(\by,\bgamma^{-j}(\by^{-j}),\delta_j) \prod_{i \neq j} \pi_i(y_i) - \Omega_j(\delta_j), \label{eq2}
\end{align}
where $(I)$ follows from the fact that $\Omega_i$ does not depend on $y_j$ and $\delta_j$ if $i \neq j$. Note that since $R(\by,\bdelta)$ is a multi-linear function of $\bdelta$ and $\Omega_i$ is a strongly convex function of $\delta_j$, the set in the right side of (\ref{eq2}) is a singleton. Hence, $\rB_j$ is well-defined; that is, it is a single-valued function (not a multi-valued function). 

\begin{lemma}\label{bestlemma}
For each $j=1,\ldots,N$, given any $\bgamma^{-j} \in \bGamma^{-j}$, we have 
$$\argmax_{\psi_j \in \Gamma_j} J^{\reg}(\bgamma^{-j},\psi_j) = \rB_j(\bgamma^{-j}).$$  
\end{lemma}

\begin{proof}
It is straightforward to prove that $\rB_j(\bgamma^{-j}) \in \argmax_{\psi_j \in \Gamma_j} J^{\reg}(\bgamma^{-j},\psi_j)$.
Let us suppose that $\xi_j \in \argmax_{\psi_j \in \Gamma_j} J^{\reg}(\bgamma^{-j},\psi_j)$. Since $\pi_j(y_j) = \frac{1}{|\sY_j|} > 0$ for all $y_j$, we must have  
$$
\xi_j(y_j) \in \argmax_{\delta_j \in \Delta_j} E\left[R^{\reg}({\bf Y},\bgamma^{-j}({\bf Y}^{-j}),\delta_j) \bigg | Y_j = y_j \right]. 
$$
Otherwise, we can construct another policy for Agent~$j$ which performs better than $\xi_j$ given $\bgamma^{-j}$.
But since $\argmax_{\delta_j \in \Delta_j} E\left[R^{\reg}({\bf Y},\bgamma^{-j}({\bf Y}^{-j}),\delta_j) \bigg | Y_j = y_j \right]$ is singleton and equal to $\rB_j(\bgamma^{-j})(y_j)$, we have $\xi_j = \rB_j(\bgamma^{-j})$, which completes the proof. 
\end{proof}

Now, we can define the best response operator $\rB: \bGamma \rightarrow \bGamma$ as follows:
$$
\rB(\bgamma) \coloneqq \left(\rB_1(\bgamma^{-1}),\ldots,\rB_N(\bgamma^{-N})\right).
$$
The following result is the key to prove the main results of this paper. In its proof, we will make use of the properties $2.$ and $3.$ of Proposition~\ref{duality}. 

\smallskip

\begin{proposition}\label{key}
Given any $\bgamma, \bxi \in \bGamma$, for each $j=1,\ldots,N$, we have 
$$
\|\rB_j(\bgamma^{-j})-\rB_j(\bxi^{-j})\|_{L_1} \leq \frac{1}{\rho_j} \sum_{i \neq j} \frac{\lambda_i(r)}{2} \, \|\gamma_i-\xi_i\|_{L_1}.
$$
\end{proposition}
\begin{proof}
Fix any $j=1,\ldots,N$. Let $\bgamma \in \bGamma$. Define $$
G_j(\bgamma^{-j},y_j,u_j) \coloneqq \sum_{\by^{-j},\bu^{-j}} r(\by,\bu) \, \prod_{i \neq j} \gamma_i(u_i|y_i) \, \pi_i(y_i). 
$$
Then, we have 
\begin{align}
&\rB_j(\bgamma^{-j})(y_j) \nonumber \\
&= \argmax_{\delta_j \in \Delta_j} \sum_{\by^{-j}} R(\by,\bgamma^{-j}(\by^{-j}),\delta_j) \prod_{i \neq j} \pi_i(y_i) - \Omega_j(\delta_j) \nonumber \\ 
&= \argmax_{\delta_j \in \Delta_j} \, \left \langle G_j(\bgamma^{-j},y_j,\cdot),\delta_j \right\rangle - \Omega_j(\delta_j).\nonumber 
\end{align}
By the property $2.$ of Proposition~\ref{duality}, we have $$
\rB_j(\bgamma^{-j})(y_j) = \nabla \Omega_j^*(G_j(\bgamma^{-j},y_j,\cdot)). 
$$
Moreover, if $\bgamma, \bxi \in \bGamma$, by property $3.$ of Proposition~\ref{duality} and by noting the fact that the $l_{\infty}$ norm $\|\cdot\|_{\infty}$ is dual norm of the $l_1$ norm $\|\cdot\|_1$ on $\Delta_j$, we obtain the following bounds:
\begin{align}
&\|\rB_j(\bgamma^{-j})(y_j)-\rB_j(\bxi^{-j})(y_j)\|_{1} \nonumber \\
&= \|\nabla \Omega_j^*(G_j(\bgamma^{-j},y_j,\cdot)) - \nabla \Omega_j^*(G_j(\bxi^{-j},y_j,\cdot)) \|_{1}  \nonumber \\
&\leq \frac{1}{\rho_j} \, \|G_j(\bgamma^{-j},y_j,\cdot)-G_j(\bxi^{-j},y_j,\cdot)\|_{\infty} \nonumber \\
&\phantom{xxxxxxxxxxxxxxxxx}\text{ } \text{(by property $3.$ of Proposition~\ref{duality})} \nonumber \\
&= \frac{1}{\rho_j} \, \sup_{u_j} |G_j(\bgamma^{-j},y_j,u_j)-G_j(\bxi^{-j},y_j,u_j)| \nonumber \\
&= \frac{1}{\rho_j} \, \sup_{u_j} \bigg|\sum_{\by^{-j},\bu^{-j}} r(\by,\bu) \, \prod_{i \neq j} \gamma_i(u_i|y_i) \, \pi_i(y_i) \nonumber \\
&\phantom{xxxxxxxxxxxxxx} - \sum_{\by^{-j},\bu^{-j}} r(\by,\bu) \, \prod_{i \neq j} \xi_i(u_i|y_i) \, \pi_i(y_i)\bigg| \nonumber \\
&= \frac{1}{\rho_j} \, \sup_{u_j} \bigg|\sum_{\by^{-j}} \big( R(\by,\bgamma^{-j}(\by^{-j}),\delta_{u_j}) \nonumber \\
&\phantom{xxxxxxxxxxxxxxxx}- R(\by,\bxi^{-j}(\by^{-j}),\delta_{u_j}) \big) \prod_{i \neq j} \pi_i(y_i) \bigg| \nonumber \\
&\leq \frac{1}{\rho_j} \, \sum_{\by^{-j}} \prod_{i \neq j} \pi_i(y_i) \left(\sum_{i \neq j} \frac{\lambda_i(r)}{2} \, \|\gamma_i(y_i)-\xi_i(y_i)\|_{1}\right) \nonumber \\
&\phantom{xxxxxxxxxxxxxxxxxxxxxxxxxxxxx}\text{ } \text{(by Lemma~\ref{localosc})} \nonumber \\
&= \frac{1}{\rho_j} \, \sum_{i \neq j} \sum_{y_i} \pi_i(y_i) \, \frac{\lambda_i(r)}{2} \|\gamma_i(y_i)-\xi_i(y_i)\|_{1} \nonumber \\
&= \frac{1}{\rho_j} \, \sum_{i \neq j} \frac{\lambda_i(r)}{2} \, \|\gamma_i-\xi_i\|_{L_1}. \nonumber
\end{align}
Hence, we have 
\begin{align}
\|\rB_j(\bgamma^{-j})&-\rB_j(\bxi^{-j})\|_{L_1} \nonumber \\
& = \sum_{y_j} \|\rB_j(\bgamma^{-j})(y_j)-\rB_j(\bxi^{-j})(y_j)\|_{1} \, \pi_j(y_j) \nonumber \\
&\leq \frac{1}{\rho_j} \sum_{i \neq j} \frac{\lambda_i(r)}{2} \, \|\gamma_i-\xi_i\|_{L_1}.\nonumber \end{align}
This completes the proof.
\end{proof}

Let us define the non-negative $N \times N$ matrix $P$ as follows:
\begin{align}
P(j,i) &= \frac{1}{\rho_j} \cdot \frac{\lambda_i(r)}{2} \,\, \text{for} \,\, i \neq j, \nonumber \\
P(j,j) &= 0 \,\, \text{for} \,\, j=1,\ldots,N. \nonumber 
\end{align}
We assume the following condition.

\smallskip

\begin{assumption}\label{as1}
The spectral radius $\alpha(P)$ of $P$ is strictly less than $1$.  
\end{assumption}

For any $\bgamma \in \bGamma$, we let $\|\bgamma\| \coloneqq \left(\|\gamma_1\|_{L_1},\ldots,\|\gamma_N\|_{L_1}\right)^T$; that is, $\|\bgamma\|$ is a column vector that consists of $L_1$ norms of $(\gamma_1,\ldots,\gamma_N)$. By Proposition~\ref{key}, we have
$$
\|\rB(\bgamma)-\rB(\bxi)\| \leq P \cdot \|\bgamma-\bxi\| \,\, \text{for any} \,\, \bgamma,\bxi \in \bGamma. 
$$
Namely, $\rB$ is a $P$-contraction on $\bGamma$ \cite{OrRh00}. The following is the first main result of this paper. 

\smallskip

\begin{theorem}\label{mainthm1}
Suppose that Assumption~\ref{as1} holds. Define iterates $\bgamma^{k+1} = \rB(\bgamma^k)$. Then $\lim_{k \rightarrow \infty} \|\bgamma^k - \bgamma^*\| = 0$, where $\bgamma^*$ is the unique fixed point of $\rB$. Moreover, $\bgamma^*$ is the optimal regularized team decision rule.
\end{theorem}
\begin{proof}
The proof of the first part can be done as in the proof of \cite[13.1.2, p. 433]{OrRh00}. But for the sake of completeness, we give the full proof of this result here. 
The proof is almost the same as Banach Fixed Point Theorem if you note the following facts about matrix $P$ under Assumption~\ref{as1}:
$$
(I-P)^{-1} = \sum_{i=0}^{\infty} P^i \geq 0, \, \sum_{i=0}^{k} P^i \leq (I-P)^{-1} ,  \lim_{k\rightarrow\infty} P^k = 0. 
$$
For any $k,m\geq0$, we have 
\begin{align}
\|\bgamma^{k+m}-\bgamma^{k}\|  &\leq \sum_{j=1}^m \|\bgamma^{k+j}-\bgamma^{k+j-1}\|  \nonumber \\
&\leq \sum_{j=1}^m P^j \, \|\bgamma^{k}-\bgamma^{k-1}\| \nonumber \\
&\leq (I-P)^{-1} \, P \, \|\bgamma^{k}-\bgamma^{k-1}\| \nonumber \\
&\leq (I-P)^{-1} \, P^k \, \|\bgamma^{1}-\bgamma^{0}\|. \nonumber
\end{align}
Since $\lim_{k\rightarrow\infty} P^k = 0$, $\{\|\bgamma^k\|\}_{k\geq1}$ is a Cauchy sequence, and so, converges to some $\bgamma^* \in \bGamma$ as $\bGamma$ is closed. Note that 
\begin{align}
\|\bgamma^*-\rB(\bgamma^*)\| &\leq \|\bgamma^*-\bgamma^{k+1}\| + \|\rB(\bgamma^{k})-\rB(\bgamma^*)\| \nonumber \\
&\leq \|\bgamma^*-\bgamma^{k+1}\| + P \, \|\bgamma^{k}-\bgamma^*\|. \nonumber
\end{align}
Since the terms in the last expression converge to zero, we have $\bgamma^*=\rB(\bgamma^*)$. Let $\bxi^*$ be another fixed point of $\rB$ in $\bGamma$. Then
$$
\|\bgamma^*-\bxi^*\| = \|\rB(\bgamma^*)-\rB(\bxi^*)\| \leq P \, \|\bgamma^*-\bxi^*\|.  
$$
Hence, $(I-P) \, \|\bgamma^*-\bxi^*\| \leq 0$. Since $(I-P)^{-1} \geq 0$, we must have $\bgamma^*=\bxi^*$. This completes the proof of the first part.

To show the global optimality of $\bgamma^*$, note that any globally optimal policy should be a fixed point of $\rB$. Otherwise, by Lemma~\ref{bestlemma}, we can construct a better policy using best response maps $\{\rB_j\}$. Since $\rB$ has an unique fixed point $\bgamma^*$, this unique fixed point should be optimal regularized team decision rule if it exists. Existence of optimal regularized team decision rule follows from \cite[Theorem 5.2]{YukselWitsenStandardArXiv}. 
\end{proof}

The next result is a corollary of Theorem~\ref{mainthm1}. It shows how the regularized optimal team decision rule $\gamma^*$ in Theorem~\ref{mainthm1} performs in the original stochastic team problem.  

\smallskip

\begin{corollary}\label{cor1}
For each $j=1,\ldots,N$, define 
$$
\beta_j \coloneqq \sup_{\delta_j \in \Delta_j} \Omega_j(\delta_j) - \inf_{\delta_j \in \Delta_j} \Omega_j(\delta_j). 
$$
Then, we have 
$$
J^* - \sum_{j=1}^N \beta_j \leq J(\bgamma^*) \leq J^*, 
$$
where $\bgamma^*$ is the optimal regularized team decision rule.
\end{corollary}

\begin{proof}
The second inequality is obvious. For the first inequality, note that 
\begin{align}
J^* &\leq J^{\reg}(\bgamma^*) + \sum_{j=1}^N \sup_{\delta_j \in \Delta_j} \Omega_j(\delta_j)  \nonumber \\  
&\leq J(\bgamma^*) + \sum_{j=1}^N \sup_{\delta_j \in \Delta_j} \Omega_j(\delta_j) - \sum_{j=1}^N \inf_{\delta_j \in \Delta_j} \Omega_j(\delta_j) \nonumber \\
&= J(\bgamma^*) + \sum_{j=1}^N \beta_j. \nonumber 
\end{align}
This completes the proof.
\end{proof}

Note that as a result of Corollary~\ref{cor1}, the solution of the regularized team problem provides an upper bound and a lower bound to the original team problem, which can be used to analyze the performance of certain numerical algorithms developed in the literature.

\subsubsection{Example 1}\label{example1}

In this example, we consider two-agent stochastic team problem and derive a necessary condition for Assumption~\ref{as1}. In two agent case, the matrix $P$ is in the following form
\begin{align}
\begin{pmatrix}
0 & \frac{1}{\rho_1} \cdot \frac{\lambda_2(r)}{2} \\
\frac{1}{\rho_2} \cdot \frac{\lambda_1(r)}{2} & 0 
\end{pmatrix} \nonumber 
\end{align}
Hence, the spectral radius of $P$ is
$$\alpha(P) = \sqrt{\frac{\lambda_2(r)\lambda_1(r)}{4\rho_2 \rho_1}}.$$
Therefore, if $\frac{\lambda_2(r)\lambda_1(r)}{4\rho_2 \rho_1} < 1$, then the results of this paper hold for two-agent case. 

Note that we can achieve the same result if we define the following operator $T:\Gamma_2 \rightarrow \Gamma_2$ as 
$$
T(\gamma_2) = \rB_2(\rB_1(\gamma_2)). 
$$
Indeed, if $\gamma_2,\xi_2 \in \Gamma_2$, then by Proposition~\ref{key}
\begin{align}
\|T(\gamma_2)-T(\xi_2)\|_{L_1} &\leq \frac{1}{\rho_2} \cdot \frac{\lambda_1(r)}{2} \|\rB_1(\gamma_2)-\rB_1(\xi_2)\|_{L_1}  \nonumber \\
&\leq \frac{1}{\rho_2} \cdot \frac{\lambda_1(r)}{2} \frac{1}{\rho_1} \cdot \frac{\lambda_2(r)}{2} \|\gamma_2-\xi_2\|_{L_1}  \nonumber \\
&= \frac{\lambda_2(r)\lambda_1(r)}{4\rho_2 \rho_1} \|\gamma_2-\xi_2\|_{L_1}. \nonumber 
\end{align}
Hence, $T$ is contraction if $\frac{\lambda_2(r)\lambda_1(r)}{4\rho_2 \rho_1} < 1$. Suppose that this is the case. Let $\gamma_2^*$ be the unique fixed point of $T$ and let $\gamma_1^* = \rB_1(\gamma_2^*)$. Then, it is straightforward to prove that $(\gamma_2^*,\gamma_1^*)$ is the unique solution in Theorem~\ref{mainthm1}. This is indeed the approach developed in \cite{Bas87} to prove the uniqueness of Nash equilibrium in two person game problems, which subsumes two person team problems.. 

\subsubsection{Example 2}\label{example2}

In this example, we consider $N$-agent model with circular dependent reward function $r$; that is, $r$ can be decomposed as follows:
\begin{align}
&r(\by,\bu) = r_1(y_1,y_2,u_1,u_2)+r_2(y_2,y_3,u_2,u_3) \nonumber \\
&\phantom{xxxxxxxxxxxxxxxxxxxx}+\cdots+r_N(y_N,y_1,u_N,u_1). \nonumber 
\end{align}
This is indeed the case if the original reward function $p$ is of the following form
\begin{align}
&p(x,\by,\bu) = p_1(x,y_1,y_2,u_1,u_2)+p_2(x,y_2,y_3,u_2,u_3) \nonumber \\
&\phantom{xxxxxxxxxxxxxxxxxxxx}+\cdots+p_N(x,y_N,y_1,u_N,u_1). \nonumber 
\end{align}
Note that, in this case, for any $j=1,\ldots,N-1$, $\rB_j$ is only a function of $\gamma_{j+1}$, and $\rB_N$ is only a function of $\gamma_1$. Therefore, the matrix $P$ is in the following form:
\scriptsize
\begin{align}
\begin{pmatrix}
0 & \frac{1}{\rho_1} \cdot \frac{\lambda_2(r)}{2} & 0 & 0 & \ldots & 0 \\
0 & 0 & \frac{1}{\rho_2} \cdot \frac{\lambda_3(r)}{2} &  0 & \vdots & 0 \\
 \vdots & \vdots & 0 & \ddots &  \vdots & \vdots \\
 0 & \vdots & \vdots & \vdots & \ddots & \vdots \\
0 & \vdots & \vdots & \vdots & 0 & \frac{1}{\rho_{N-1}} \cdot \frac{\lambda_N(r)}{2}\\
\frac{1}{\rho_{N}} \cdot \frac{\lambda_1(r)}{2} & 0 & 0 & \ldots & 0 & 0
\end{pmatrix} \nonumber 
\end{align}
\normalsize
Hence, the spectral radius of $P$ is
$$\alpha(P) = \sqrt[N]{\frac{\lambda_N(r)\ldots\lambda_1(r)}{2^N \rho_N \ldots \rho_1}}.$$
This implies that if $\frac{\lambda_N(r)\ldots\lambda_1(r)}{2^N \rho_N \ldots \rho_1}< 1$, then results of this paper hold for circular dependent case. 

As in Example~\ref{example1}, we can achieve the same result if we define the following operator $T:\Gamma_N \rightarrow \Gamma_N$ as 
$$
T(\gamma_N) = \rB_{N}(\rB_{N-1}(\ldots(\rB_1(\gamma_N)))). $$
Then, if $\gamma_N,\xi_N \in \Gamma_N$, it is straightforward to show that 
\begin{align}
\|T(\gamma_N)-T(\xi_N)\|_{L_1} \leq \frac{\lambda_N(r)\ldots\lambda_1(r)}{2^N \rho_N \ldots \rho_1} \|\gamma_N-\xi_N\|_{L_1}. \nonumber 
\end{align}
Hence, $T$ is contraction if $\frac{\lambda_N(r)\ldots\lambda_1(r)}{2^N \rho_N \ldots \rho_1} < 1$. Suppose that this is the case. Let $\gamma_N^*$ be the unique fixed point of $T$ and define recursively $\gamma_1^* =\rB_1(\gamma_N^*), \ldots, \gamma_{N-1}^* =\rB_{N-1}(\gamma_{N-2}^*)$. Then, it is straightforward to show that $(\gamma_N^*,\ldots,\gamma_1^*)$ is the unique solution in Theorem~\ref{mainthm1}. This approach was again first developed in \cite{Bas87} to prove the uniqueness of Nash equilibrium in $N$ person game problems with circular dependent reward functions.

\section{Asynchronous Iterative Algorithm}\label{asyiteration}

In this section, we propose an asynchronous iterative algorithm for computing the optimal team decision rule $\gamma^*$ and prove its convergence. This algorithm was first introduced in \cite{Ber83} to find fixed points of vector-valued functions. A similar asynchronous iterative algorithm was introduced in \cite{LiBa87} to compute Nash equilibrium in games; that is, the objectives of agents are different.  

In this algorithm, at each iteration, Agent~$j$ can be in one of three possible states $\{\textit{compute},\textit{transmit},\textit{idle}\}$. In the compute state, Agent~$j$ computes a new policy $\gamma_j$ using $\rB_j$ and available policies of other agents stored in its memory. In the transmit state, Agent~$j$ sends its latest policy to one or more agents. In the idle state, Agent~$j$ does nothing. It is assumed that an agent can receive transmission from other agents while computing or transmitting. We have the following assumption about the timing of this algorithm.

\smallskip

\begin{assumption}\label{as2}
For each Agent~$j$ and iteration time $t$, there exists $t'>t$ such that between $t$ and $t'$, Agent~$j$ should do at least one computation and should do transmission to every other agent. 
\end{assumption}

In this algorithm, at each iteration time $t$, every Agent~$j$ stores $N$-tuple of policies $\bgamma^{(t)}(j) = (\gamma_1^{(t)}(j),\ldots,\gamma_N^{(t)}(j))$ in its memory, where $\gamma_i^{(t)}(j)$ is the latest transmission from Agent~$i$ to Agent~$j$ and $\gamma_j^{(t)}(j)$ is Agent~$j$'s latest own policy estimate. Note that memory contents of different agents at the same time can be different. This is also true for initial time; that is, at the initialization of the algorithm, agents can use different $N$-tuple of policies. 

The rules according to which the memory contents are updated as follows:
\begin{itemize}
\item[(1)] If $t$ is a transmission time from Agent~$i$ to Agent~$j$, the policy $\gamma_i^{(t-1)}(i)$ is sent to Agent~$j$ and Agent~$j$ updates its $i^{th}$ policy as follows $\gamma_i^{(t)}(j)=\gamma_i^{(t-1)}(i)$.
\item[(2)] If $t$ is a computation time for Agent~$i$, Agent~$i$ updates its own policy using $\rB_i$; that is, $\gamma_i^{(t)}(i) = \rB_i\left(\gamma_1^{(t-1)}(i),\ldots,\gamma_{i-1}^{(t-1)}(i),\gamma_{i+1}^{(t-1)}(i),\ldots,\gamma_{N}^{(t-1)}(i)\right).$
\end{itemize}

Now, we can state our second main result, which is about convergence of the above asynchronous iterative algorithm for computing optimal regularized team decision rule $\bgamma^*$. 

\smallskip

\begin{theorem}\label{mainthm2}
Suppose that Assumption~\ref{as1} and Assumption~\ref{as2} hold. For $j=1,\ldots,N$, let $\bgamma^{(0)}(j)$ be the initial memory content of Agent~$j$. For each $t$, let $\bgamma^{(t)}$ be defined as $\gamma_j^{(t)} = \gamma_j^{(t)}(j)$. Then the iterative asynchronous algorithm converges to the unique fixed point of $\rB$; that is
$$
\lim_{t\rightarrow\infty} \|\bgamma^{(t)}-\bgamma^*\|=0.
$$
\end{theorem}

\begin{proof}
Note that there exists $\alpha>0$ such that $\|\bgamma^{(0)}(j) - \bgamma^*\| \leq \alpha$ for all $j=1,\ldots,N$. Since the spectral radius $\alpha(P)$ of $P$ is less than $1$, by \cite[Lemma, p. 231]{Bau78}, there exists a positive vector $v \in \R^N$ and positive constant $w<1$ such that $P \cdot v \leq w \cdot v$. For each $k=0,1,2,\ldots$, let us define 
$$
\bGamma^k \coloneqq \{\bgamma \in \bGamma: \|\bgamma - \bgamma^*\| \leq \alpha w^k \cdot v\}. 
$$
We prove that the sequence of subsets 
$\{\bGamma^k\}_{k\geq0}$ of $\bGamma$ satisfies the following properties:
\begin{itemize}
\item[(a)] If $\{\bgamma^{(k)}\}$ is a sequence in $\bGamma$ such that $\bgamma^{(k)} \in \bGamma^k$ for all $k$, then 
$$
\lim_{k\rightarrow\infty} \|\bgamma^{(k)}-\bgamma^*\| = 0. 
$$
\item[(b)] For all $k=0,1,\ldots$ and $j=1,\ldots,N$, if $\bgamma \in \bGamma^k$, then $(\bgamma^{-j},\rB_j(\bgamma^{-j})) \in \bGamma^k$.
\item[(c)] For all $k=0,1,\ldots$ and $j=1,\ldots,N$, if $\bgamma, \bxi \in \bGamma^k$, then $(\bgamma^{-j},\xi_j) \in \bGamma^k$.
\item[(d)] For all $k=0,1,\ldots$, if $\bgamma(1) \in \bGamma^k,\ldots,\bgamma(N) \in \bGamma^k$, then $$\left(\rB_1(\bgamma^{-1}(1)),\ldots,\rB_N(\bgamma^{-N}(N))\right) \in \bGamma^k.$$ 
\end{itemize}
The proof then follows from \cite[Proposition, p. 114]{Ber83}.

To show property (a), let $\{\bgamma^{(k)}\}$ be a sequence in $\bGamma$ such that $\bgamma^{(k)} \in \bGamma^k$ for all $k$; that is,
$$
\|\bgamma^{(k)} - \bgamma^*\| \leq \alpha w^k \cdot v. 
$$
Since $w < 1$, $\lim_{k\rightarrow\infty} \|\bgamma^{(k)}-\bgamma^*\| = 0$. Hence property (a) holds. For property (b), let $\bgamma \in \bGamma^k$; that is, $\|\bgamma - \bgamma^*\| \leq \alpha w^k \cdot v$. Then we have 
\begin{align}
&\|\rB(\bgamma)-\bgamma^*\| = \|\rB(\bgamma)-\rB(\bgamma^*)\| \leq P \cdot \|\bgamma-\bgamma^*\| \nonumber \\
&\leq P \cdot \alpha w^k \cdot v \leq \alpha w^{k+1} \cdot v \nonumber 
\end{align}
Hence $\|\rB_j(\bgamma^{-j})-\gamma_j^*\|_{L_1} \leq \alpha w^{k+1} v_j \leq \alpha w^{k} v_j$ for any $j=1,\ldots,N$. This implies that $(\bgamma^{-j},\rB_j(\bgamma^{-j})) \in \bGamma^k$. Therefore, property (b) holds. Property (c) is straightforward to prove, and so, we omit the details. To show property (d), let $\bgamma(1) \in \bGamma^k,\ldots,\bgamma(N) \in \bGamma^k$. Then, for any $j=1,\ldots,N$, we have  
\begin{align}
&\|\rB(\bgamma(j))-\bgamma^*\| = \|\rB(\bgamma(j))-\rB(\bgamma^*)\| \leq P \cdot \|\bgamma(j)-\bgamma^*\| \nonumber \\
&\leq P \cdot \alpha w^k \cdot v \leq \alpha w^{k+1} \cdot v \nonumber 
\end{align}
Hence $\|\rB_j(\bgamma^{-j}(j))-\gamma_j^*\|_{L_1} \leq \alpha w^{k+1} v_j \leq \alpha w^{k} v_j$ for any $j=1,\ldots,N$. In other words, $$\|\left(\rB_1(\bgamma^{-1}(1)),\ldots,\rB_N(\bgamma^{-N}(N))\right)-\bgamma^*\| \leq \alpha w^k \cdot v.$$  
This implies that $$\left(\rB_1(\bgamma^{-1}(1)),\ldots,\rB_N(\bgamma^{-N}(N))\right) \in \bGamma^k.$$
Hence, property (d) holds. This completes the proof in view of \cite[Proposition, p. 114]{Ber83}. 
\end{proof}

\section{Extension to Continuous Observation Spaces}

Since the main motivation of the paper is to compute the optimal regularized team decision rule, we therefore assume that the observation spaces are finite. However, one can do the same analysis for the stochastic teams with Borel observation spaces $\{\sY_i, i=1,\ldots,N\}$ under the following absolute continuity conditions on the observation channels:
\smallskip

\begin{itemize}
\item[\textbf{(AC)}] For each $i=1,\ldots,N$, there exists a probability measure $\pi_i$ on $\sY_i$ such that $W_i(dy_i|x,\bu^{^{(1:i-1)}})$ is absolutely continuous with respect to $\pi_i$ for any $(x,\bu^{^{(1:i-1)}})$, and the corresponding density function is $f_i(y_i,x,\bu^{^{(1:i-1)}})$. 
\end{itemize}

\smallskip

Under condition \textbf{(AC)}, one can reduce the dynamic stochastic team problem to the independent static one by incorporating the $f_i$ terms into the original reward function $p$. That is, for a fixed choice of $\bgamma$, the reward function $J(\bgamma)$ can be written as

\small
\begin{align}
&J(\bgamma) \nonumber \\
&= \int \sum_{\bu} p(x,\by,\bu)  \left( \prod_{i=1}^N \gamma_i(u_i|y_i) f_i(y_i,x,\bu^{^{(1:i-1)}}) \pi_i(dy_i)  \right)  \hspace{-3pt} \rP(dx) \nonumber \\
&= \int \, \sum_{\bu} r(\by,\bu)\,  \prod_{i=1}^N   \gamma_i(u_i|y_i) \, \pi_i(dy_i), \nonumber
\end{align}
where
\begin{align}
r(\by,\bu) &\coloneqq \int_{\sX} \biggl[ p(x,\by,\bu) \prod_{i=1}^N  f_i(y_i,x,\bu^{^{(1:i-1)}}) \biggr] \rP(dx) \nonumber
\end{align} 
\normalsize
is the new reward function. Now, the observations can be regarded as independent. In this case, any policy $\gamma_i: \sY_i \rightarrow \Delta_i$ of Agent~$i$ has the following $L_1$ norm:
$$
\|\gamma_i\|_{L_1} \coloneqq \int_{\sY_i} \|\gamma_i(\cdot|y_i)\|_1 \, \pi_i(dy_i), 
$$
which is very similar to the $L_1$-norm in the finite observation setting. The following theorem is the main result of this section. 

\begin{theorem}\label{auxthm}
Suppose that observation spaces $\{\sY_i, i=1,\ldots,N\}$ are Borel and \textbf{(AC)} holds. Suppose also that Assumption~\ref{as1} is true. Define iterates $\bgamma^{k+1} = \rB(\bgamma^k)$. Then $\lim_{k \rightarrow \infty} \|\bgamma^k - \bgamma^*\| = 0$, where $\bgamma^*$ is the unique fixed point of $\rB$. Moreover, $\bgamma^*$ is the optimal regularized team decision rule.

Suppose further that Assumption~\ref{as2} holds. For $j=1,\ldots,N$, let $\bgamma^{(0)}(j)$ be the initial memory content of Agent~$j$. For each $t$, let $\bgamma^{(t)}$ be defined as $\gamma_j^{(t)} = \gamma_j^{(t)}(j)$. Then the iterative asynchronous algorithm converges to the unique fixed point of $\rB$; that is
$$
\lim_{t\rightarrow\infty} \|\bgamma^{(t)}-\bgamma^*\|=0.
$$
\end{theorem}

\begin{proof}
We can apply exactly the same analyses as in Section~\ref{bestresponse} and Section~\ref{asyiteration} to prove the theorem. Therefore, we omit the details. 
\end{proof}

Note that we cannot extend Theorem~\ref{mainthm1} and Theorem~\ref{mainthm2} to the stochastic teams with continuous action spaces. Indeed, if we suppose that action spaces $\{\sU_i, i=1,\ldots,N\}$ are Borel, then  $\Delta_i$'s (sets of probability distributions on $\sU_i$'s) become infinite dimensional. In this case, to obtain similar results, one needs to extend duality of strong convexity and smoothness (i.e., Proposition~\ref{duality}) to the functions defined on infinite dimensional spaces $\Delta_i$, which is a highly non-trivial result to prove. This is indeed a future research direction to pursue.

\section{Conclusion}\label{conc}

In this paper, we introduced regularized stochastic team problems. We established that the best response operator has an unique fixed point and this unique fixed point is the optimal regularized team decision rule. Then, we introduced an asynchronous iterative algorithm for the computation of this unique fixed point.  

One interesting future direction is to study regularized stochastic team problems with abstract observation and action spaces. In this case, to obtain similar results, one needs to extend duality of strong convexity and smoothness to the functions defined on infinite dimensional spaces such as the set of probability measures on abstract spaces.

\section*{Appendix}

\subsection{Proof of Lemma~\ref{localosc}}\label{app1}

Given $\by$, we have 
\begin{align}
&|R(\by,\bdelta)-R(\by,\bxi)| \nonumber \\
&= \left| \sum_{\bu} r(\by,\bu) \prod_{i=1}^N \delta_i(u_i) -  \sum_{\bu} r(\by,\bu) \prod_{i=1}^N \xi_i(u_i) \right| \nonumber \\  &\leq \sum_{j=1}^N \bigg| \sum_{\bu} r(\by,\bu) \prod_{i=1}^{N-j+1} \delta_i(u_i) \, \prod_{k=1}^{j-1} \xi_{N-j+1+k}(u_{N-j+1+k}) \nonumber \\
&\phantom{xxxxxxx}-\sum_{\bu} r(\by,\bu) \prod_{i=1}^{N-j} \delta_i(u_i) \, \prod_{k=1}^{j} \xi_{N-j+k}(u_{N-j+k}) \bigg| \nonumber \\
&= \sum_{j=1}^N \left| \sum_{u} F_j(\by,u) \, \delta_{N-j+1}(u) - \sum_{u} F_j(\by,u) \, \xi_{N-j+1}(u) \right|, \label{eq1}
\end{align}
where
\begin{align}
&F_j(\by,u_{N-j+1}) \nonumber \\
&\coloneqq \sum_{\bu^{-(N-j+1)}} r(\by,\bu) \prod_{i=1}^{N-j} \delta_i(u_i) \, \prod_{k=1}^{j-1} \xi_{N-j+k}(u_{N-j+1+k}). \nonumber 
\end{align}
Note that 
$$
\sup_{\by} \left[ \sup_{u} F_j(\by,u) - \inf_{u} F_j(\by,u) \right] \leq \lambda_{N-j+1}(r), 
$$
for all $j=1,\ldots,N$. 

Before we conclude the proof, let us recall the following fact about $l_1$ norm on the set probability distributions on finite sets \cite[p. 141]{Geo11}. Suppose that we have a real valued function $F$ on a finite set $\sU$. Let $\lambda(F) \coloneqq \sup_{u \in \sU} F(u) - \inf_{u \in \sU} F(u)$. Then, for any pair of probability distributions $\mu,\nu$ on $\sU$, we have 
$$
\left|\sum_{u} F(u) \, \mu(u) - \sum_{u} F(u) \, \nu(u) \right| \leq \frac{\lambda(F)}{2} \, \|\mu-\nu\|_{1}.
$$
Using this fact, we have
\begin{align}
(\ref{eq1}) &\leq \sum_{j=1}^N \frac{\lambda_{N-j+1}}{2} \, \|\delta_{N-j+1}-\xi_{N-j+1}\|_{1} \nonumber \\
&=\sum_{i=1}^N \frac{\lambda_i(r)}{2} \, \|\delta_i-\xi_i\|_{1}. \nonumber 
\end{align}
This completes the proof.

\section{Acknowledgements}

The author is grateful to Professor Serdar Y\"{u}ksel and Tamer Ba\c{s}ar for their constructive comments.

%\bibliographystyle{IEEEtran}
%\bibliography{references,SerdarBibliography}


\end{document}